\documentclass[11pt,a4paper]{article}
\usepackage[a4paper]{anysize}

\usepackage[english]{babel}
\usepackage[latin1]{inputenc}
\usepackage[T1]{fontenc} % damit Umlaute richtig dargestellt werden

\usepackage{amssymb}
\usepackage{amsmath}
\usepackage{amsfonts
           ,amsthm
           ,enumerate
           ,bbm %This package is for the indicator function letter, the 1...
           ,xcolor
           }

\numberwithin{equation}{section}
\usepackage[nodayofweek]{datetime}

\usepackage{hyperref}
\newcommand{\IE}{\mathbb{E}}

\newcommand{\IP}{\mathbb{P}}

\newcommand{\IR}{\mathbb{R}}

\newcommand{\cE}{\mathcal{E}}
\newcommand{\cF}{\mathcal{F}}

\newcommand{\cH}{\mathcal{H}}

\newcommand{\cS}{\mathcal{S}}

\newcommand{\ud}{\mathrm{d}}
\newcommand{\uds}{\mathrm{d}s}

\newcommand{\udws}{\mathrm{d}W_s}

\newcommand{\1}{\mathbbm{1}}
\newcommand{\bit}{\begin{itemize}}
\newcommand{\eit}{\end{itemize}}
\theoremstyle{plain}
\newtheorem{theo}{Theorem}[section]
\newtheorem{lemma}[theo]{Lemma}

\newtheorem{coro}[theo]{Corollary}

\newtheorem{defi}[theo]{Definition}

\newtheorem{remark}[theo]{Remark}

\newtheorem{assump}[theo]{Assumption}

\title{A note on comonotonicity and positivity of the control components of decoupled quadratic FBSDE}

\author{
	\normalsize Gon\c calo Dos Reis \\[8pt]
        \small  Technische Universit\"at Berlin \\
        \small  10623 Berlin, Germany \\
        \small  and \\
	\small  CMA/FCT/UNL \\
        \small  2829-516 Caparica, Portugal \\
        \small  dosreis@math.tu-berlin.de
\and
        \normalsize  Ricardo J.N. dos Reis \\[8pt]
        \small  IDMEC, Instituto Superior T\'ecnico\\ 
        \small  Technical University Lisbon  \\
        \small  Av. Rovisco Pais \\
        \small  1049-001 Lisboa\\
        \small  Portugal \\
        \small  ricardo.reis@ist.utl.pt 
\vspace*{0.8cm}
}

\date{ \currenttime, \ddmmyyyydate\today}
\begin{document}

\selectlanguage{english}

\maketitle

\begin{abstract}
In this small note we are concerned with the solution of Forward-Backward Stochastic Differential Equations (FBSDE) with drivers that grow quadratically in the control component (quadratic growth FBSDE or qgFBSDE). The main theorem is a comparison result that allows comparing componentwise the signs of the control processes of two different qgFBSDE. As a byproduct one obtains conditions that allow establishing the positivity of the control process.
\end{abstract}
%Primary: 60H10; % Stochastic ordinary differential equations
%Secondary: 60J60% Diffusion processes
%, 60F10% Large deviations
%, 35B20% Perturbations
%

\section{Introduction}
This small note is concerned with forward-backward stochastic differential equations (BSDEs) in the Brownian framework, i.e. equations following, for some measurable functions $b$, $\sigma$, $f$ and $g$, the dynamics 
\begin{align*}
X_s^{t,x}&=x+\int_t^s b(r,X^{t,x}_r)\ud r+\int_t^s \sigma(r,X^{t,x}_r)\ud W_r,\\
Y^{t,x}_s &=g(X^{t,x}_T) +\int_s^T f(r,X^{t,x}_r,Y^{t,x}_r,Z^{t,x}_r)\ud s-\int_t^T Z^{t,x}_r\ud W_r,
\end{align*}
where $W$ a $d$-dimensional Brownian motion, $(t,x)\in[0,T]\times\IR^m$ and $s\in[t,T]$. The function $f$ is called generator or driver while $g$ is named the terminal condition function. The solution of the FBSDE is the triple of adapted processes $(X,Y,Z)$; $Z$ is called the control process.

In the last 30 years much attention has been given to this type of equations due
to their importance in the fields of optimal control and finance. The standard
theory of FBSDE is formulated under the canonical Lipschitz assumption (see for
example \cite{EPQ} and references), but in many financial problems drivers $f$
which have quadratic growth in the control component appear i.e.~when $f$
satisfies a growth condition of the type $|f(t,x,y,z)|\leq C(1+|y|+|z|^2)$. The
particular relation between FBSDE with drivers of quadratic growth in the
control component (qgFBSDE) and the field of finance, stochastic control and
parabolic PDE can be illustrated by the works  \cite{HIM2005}, \cite{HPdR10},
\cite{EPQ} and references therein.

One of the fundamental results in BSDE or FBSDE theory is the so called
comparison theorem that allows one to compare the $Y$ components of of the
solution of two BSDEs. In rough, given a terminal condition function $g^i$, a
driver $f^i$ and the corresponding FBSDE solution $(X,Y^i,Z^i)$ for
$i\in\{1,2\}$, if $g^1$ dominates $g^2$ and $f^1$ dominates $f^2$ in some sense
then this order relation is expected to carry over to the $Y$ components, i.e.
$Y^1$ dominates $Y^2$ in some sense. 
Such a result is however not possible for the control components $Z^i$. In this short note we give a type of comparison result for the control components $Z$, a so called comonotonicity result. This result allows one to compare the signs of the control processes $Z^1$ and $Z^2$ componentwise and as a side product one finds sufficient conditions to establish the positivity of the control process for a single FBSDE.

This type of results can be useful in several situations, for instance in the numerics for such equations, since they allow to establish a priori heuristics that can improve the quality of the numerical approximation. This point of view is pertinent as the applications of FBSDE extend to the field of fluid mechanics (see \cite{freidosreis2011}).

A possible application of the results presented in this note lies in the problematic of showing the existence (and smoothness) of marginal laws of $Y$ which are absolutely continuous with respect to the Lebesgue measure. This type of analysis involves showing the strict positivity of the Malliavin variance (in rough the $Z$ component) of the solution of the FBSDE, (see e.g. \cite{MR2134722}). The results in \cite{MR2134722} were established for FBSDE whose driver function satisfies a standard Lipschitz condition in its spatial components and it is not possible to adapt the proof to cover the qgFBSDE setting of this work.

From another point of view, the comonotonicity result is an interesting result in the context of economic models of equilibrium pricing when analyzed in the qgFBSDE framework. In such framework the equilibrium market price of risk can be characterized in terms of the control process of the solution to a qgFBSDE. The difficulty is that the individual optimization problems underlying the characterization of the equilibrium requires the equilibrium volatility (the $Z$ component of the solution to a certain qgFBSDE) to satisfy an exponential integrability condition as well as a positivity condition. Since the results of \cite{MR2134722} cannot be applied or adapted to the qgFBSDE setting, the comonotonicity result presented here (and its corollary) provides conditions that ensure the positivity of the relevant process and hence may prove to be very useful in equilibrium analysis. An example of such type of problems can be found for example in \cite{HPdR10}.

The results of this work originate in \cite{05CKW} where the authors give a comonotonicity result for FBSDE satisfying a standard Lipschitz condition and where the driver function is independent of the diffusion process $X$. In \cite{reis2011} the author extended the results of \cite{05CKW} to the qgFBSDE setting but was not able to include the dependence on $X$ in the driver. The dependence of $f$ in $X$ is something that is quite common in the financial framework and that makes the applicability of \cite{reis2011} limited. This short note presents a full generalization of the results of \cite{reis2011} where the driver is now allowed to depend on $X$, this makes the conditions and analysis more involved but makes the result general enough that it can now be ``broadly'' applied to the standard financial setting where the driver $f$ almost always depends on the underlying diffusion $X$.

The note is organized as follows: In Section 2 we introduce some notation and recall some known results. The main results are then stated and proved in Section 3.

\section{Preliminaries}

Throughout fix $T>0$. We work on a canonical Wiener space $(\Omega, \cF,  \IP)$ carrying a $d$-dimensional Wiener process $W = (W^1,\cdots, W^d)$ restricted to the time interval $[0,T]$ and we denote by $\cF=(\cF_t)_{t\in[0,T]}$ its natural filtration enlarged in the usual way by the $\IP$-zero sets. 

Let $p\geq 2$, then we denote by $\cS^p(\IR^m)$ the space of all measurable processes $(Y_t)_{t\in[0,T]}$ with values in $\IR^m$ normed by $\| Y \|_{\cS^p} = \IE[\sup_{t \in [0,T]}|Y_t|^p ]^{{1}/{p}}$ and by $\cS^\infty(\IR^m)$ its subspace  of bounded measurable processes. We also denote by $\cH^p(\IR^m)$ the space of all progressively measurable processes $(Z_t)_{t\in[0,T]}$ with values in $\IR^m$ normed by $\|Z\|_{\cH^p} = \IE[\big( \int_0^T |Z_s|^2 \ud s \big)^{p/2} ]^{{1}/{p}}$. 

For vectors $x = (x^1,\cdots, x^m)\in \IR^m$ we write $|x| = (\sum_{i=1}^m (x^i)^2)^{\frac{1}{2}}$. $\nabla$ denotes the canonical gradient operator and for a function $h(x,y):\IR^m\times\IR^d\to \IR$ we write $\nabla_x h$ or $\nabla_y h$ to refer to the first derivatives with relation to $x$ and $y$ respectively.

%\subsection{The equations}
We work with decoupled systems of forward and backward stochastic differential equations (FBSDE) for $(t,x)\in[0,T]\times\IR^m$ and $s\in[t,T]$
\begin{align}
\label{sde}
X_s^{t,x}&=x+\int_t^s b(r,X^{t,x}_r)\ud r+\int_t^s \sigma(r,X^{t,x}_r)\ud W_r,\\
\label{bsde}
Y^{t,x}_s &=g(X^{t,x}_T) +\int_s^T f(r,X^{t,x}_r,Y^{t,x}_r,Z^{t,x}_r)\ud s-\int_t^T Z^{t,x}_r\ud W_r,
\end{align}
for some measurable functions $b$, $\sigma$, $g$ and $f$. 

We now state our assumptions.
\begin{assump}\label{H1}
The function $b:[0,T]\times\IR^m\to \IR^m$ and $\sigma:[0,T]\times\IR^m\to \IR^{m\times d}$ are continuously differentiable in space with derivatives uniformly bounded by a constant $K$ and are $\frac12$-H\"older continuous in time. $\sigma$ is uniformly elliptic and $|b(\cdot,0)|$ and $|\sigma(\cdot,0)|$ are uniformly bounded.  

$g:\IR^m\to\IR$ is bounded, continuously differentiable with bounded derivatives. $f$ is a continuously differentiable function in space, uniformly continuous in the time variable and satisfies for some $M>0$ for all $(t,x,y,z)\in[0,T]\times \IR^m\times \IR\times \IR^d$, $|f(t,x,y,z)|\leq M (1+|y|+|z|^2)$ as well as 
\begin{align*}
|\nabla_x f(t,x,y,z)|\leq M (1+|y|+|z|^2),\quad
|\nabla_y f(t,x,y,z)|\leq M, \quad
|\nabla_z f(t,x,y,z)|\leq M (1+|z|).
\end{align*}
\end{assump}

\begin{assump}
\label{H2}
The spatial derivatives $\nabla b$, $\nabla \sigma$ and $\nabla g$ satisfy a standard Lipschitz condition in their spatial variables with Lipschitz constant $K$. 

$\nabla_y f$ satisfies a standard Lipschitz condition with Lipschitz constant $K$ and 
for all $t\in[0,T]$, $x,x'\in\IR^m$, $y,y'\in\IR$ and $z,z'\in\IR^d$ it holds that 
\begin{align*}
&|\nabla_x  f(t,x,y,z)-\nabla_x f(t,x',y',z')|
\\
&\hspace{1cm}
\leq K\big(1+|z|+|z'|\big)\big\{
(1+|z|+|z'|\big)|x-x'|+|y-y'|+|z-z'|\big\},\\
&|\nabla_z  f(t,x,y,z)-\nabla_z f(t,x',y',z')|
\\
&\hspace{1cm}
\leq K \big\{(1+|z|+|z'|) |x-x'|+|y-y'|+|z-z'|\big\},
\end{align*}
%\end{itemize}
\end{assump}
The next theorem compiles several results found throughout \cite{AIdR07}, \cite{IdR2010} and \cite{reis2011}.
\begin{theo}\label{compilationtheorem}
Let Assumption \ref{H1} hold then for any $p\geq 2$ and $(t,x)\in[0,T]\times\IR$ there exists a unique solution $\Theta^{t,x}=(X^{t,x},Y^{t,x},Z^{t,x})$ of FBSDE \eqref{sde}-\eqref{bsde} in the space $\cS^p\times\cS^\infty\times\cH^p$ and\footnote{BMO refers to the class of Bounded mean oscillation martingales, see \cite{IdR2010} or \cite{kazamaki} for more details.} $\int_0^\cdot Z\ud W \in BMO$.

The variational process of $\Theta^{t,x}$ exists and satisfies for $s\in[t,T]$
\begin{align}
\label{nablasde}
\nabla_x X_s^{t,x}&=I_d+\int_t^s \nabla_x b(r,X^{t,x}_r)\nabla_x X^{t,x}_r\ud r+\int_t^s \nabla_x\sigma(r,X^{t,x}_r)\nabla_x X^{t,x}_r\ud W_r,\\
\label{nablabsde}
\nabla_x Y^{t,x}_s &=\nabla_x g(X^{t,x}_T)\nabla_x X_T^{t,x} +\int_s^T \langle (\nabla f)(r,\Theta^{t,x}_r),\nabla_x \Theta^{t,x}_r \rangle\ud s-\int_t^T \nabla_x Z^{t,x}_r\ud W_r.
\end{align}
The triple $\Theta^{t,x}$ is Malliavin differentiable and its Malliavin
derivatives are given by $D \Theta^{t,x} = (D X^{t,x},DY^{t,x},DZ^{t,x})$. The
process $(Z_s^{t,x})_{s\in[t,T]}$ has continuous paths, $Z^{t,x}\in\cS^p$ and
for $0\leq t\leq u\leq s\leq T$ the following representation holds
\begin{align}
\label{representation}
D_s Y^{t,x}_s = Z^{t,x}_s,\ \IP\text{-}a.s. \quad \textrm{ and } \quad D_u Y^{t,x}_s = \nabla_x Y^{t,x}_s (\nabla_x X^{t,x}_u)^{-1} \sigma(u,X^{t,x}_u),\ \IP\text{-}a.s.
\end{align}
There exists a continuous function $u:[0,T]\times\IR^m\to\IR$ such that for all $(t,x)\in[0,T]\times \IR^m$ and $s\in[t,T]$ it holds that $Y^{t,x}_s=u(s,X_s^{t,x})$  $\IP$-a.s.. 

Under Assumption \ref{H2} the function $u$ is continuously differentiable in its spatial variables and $Z^{t,x}_s=(\nabla_x u)(s,X_s^{t,x})\sigma(s,X^{t,x}_s)$ $\IP$-a.s. for all $0\leq t\leq s\leq T$ and $x\in\IR^m$.
\end{theo}
\begin{proof}
Existence and uniqueness of the solution is quite standard either for the SDE (e.g. \cite{Protter2005}) or for the BSDE (see e.g. Theorem 1.2.12 and Lemma 1.2.13 in \cite{reis2011}). 

The variational differentiability and representation formulas as well as the path continuity of $Z$ follow from Theorems 2.8, 2.9 and 5.2 in \cite{IdR2010} (or Theorems 3.1.9, 3.2.4 and 4.3.2 of \cite{reis2011}). We emphasize that due to the continuity of the involved processes, the representation formulas \eqref{representation} hold $\IP$-a.s. for all $t\in[0,T]$ and not just $\IP\otimes \textrm{Leb}$-a.a.

Lastly, the Markov property of the $Y$ process is rather standard (see Theorem
4.1.1 of \cite{reis2011}).  The differentiability assumptions on the driver and
terminal condition function (Assumption \ref{H2}) ensure that the function $u$
is continuously differentiable in the spatial variables. A detailed proof of
this can be found either in Theorem 7.7 in \cite{10AIdR} or Theorem 4.1.2 in
\cite{reis2011}.
\end{proof}

\section{A comonotonicity result for quadratic FBSDE}

In this section we work with a $d$-dimensional Brownian motion $W$ on the time interval $[0,T]$ for some positive finite $T$. Throughout let $(t,x)\in [0,T]\times \IR^m$. Our standing assumption for this section is as follows.
\begin{assump}\label{H}
Let Assumptions \ref{H1} and \ref{H2} hold. Assume that $m=1$ and $d\geq 1$
\end{assump}
\begin{remark}
\label{caseofmdiff1}
We note that it is possible to write the results of this section for multidimensional SDE systems (i.e.~when $m\geq 1$) under the assumption that $\sigma$ is a square diagonal matrix  and the system of forward equations is fully decoupled. There are many applications where such an assumption takes place (e.g.  \cite{HPdR10}). We write these result with $m=1$ to simplify the presentation of this short note.
\end{remark}

For each $i\in\{1,2\}$ we define the SDE \eqref{sde} with $b_i$ and $\sigma_i$ and BSDE \eqref{bsde} with terminal condition and driver given by $g_i$ and $f_i$. We denote the respective solution of the system by $(X^{t,x,i}_s,Y^{t,x,i}_s,Z^{t,x,i}_s)_{s\in[t,T]}$ valued in $\IR\times\IR\times\IR^d$ for $(t,x,i)\in[0,T]\times\IR\times\{1,2\}$.

We define the vector-product operator, ``$\odot$'', as $\odot:\IR^d\times\IR^d\to\IR^d$ such that  
\begin{align}\label{odotoperator}
a\odot b= (a_1b_1, \ldots, a_d b_d),\qquad \textrm{for any } a=(a_1,\cdots,a_d),b=(b_1,\cdots,b_d)\in \IR^d.
\end{align}
With the convention that $a\odot b\geq 0$ means that for each $i\in\{1,\ldots,d\}$, $a_i b_i\geq 0$.

The aim of this section is to explore conditions such the following statement holds
\[Z^{t,x,1}_s \odot Z_s^{t,x,2} \geq 0,\quad 
\IP\text{-a.s.},\quad \textrm{for any } (t,x)\in[0,T]\times\IR\textrm{ and }s\in[t,T].
\]

\begin{defi}[Comonotonic functions]
We say that two measurable functions $g,h:\IR\to\IR$ are comonotonic if they are \emph{monotone} and \emph{have the same type of monotonicity}, i.e.~if $g$ is increasing or decreasing then $h$ is also increasing or decreasing respectively. We say that $g$ and $h$ are strictly comonotonic if they are comonotonic and strictly monotonic.
\end{defi}
We now state our main theorem.
\begin{theo}\label{comono-theo-1}
Let Assumption \ref{H} hold and for $(t,x)\in[0,T]\times\IR$ define $(X^{t,x,i},Y^{t,x,i},Z^{t,x,i})$ as the unique solution of FBSDE  (\ref{sde})-(\ref{bsde}) for $i\in\{1,2\}$. Suppose that $x\mapsto g_i(x)$ and $x\mapsto f_i(\cdot,x,\cdot,\cdot)$ are comonotonic for all $i\in\{1,2\}$ and further, that $g_1,g_2$ are also comonotonic\footnote{This implies that $x\mapsto f_1(\cdot,x,\cdot,\cdot)$ and $x\mapsto f_2(\cdot,x,\cdot,\cdot)$ are comonotonic as well.}. If it holds for all $s\in[t,T]$ that
\begin{align}
\label{sigmaineq}
\sigma_1(s,X^{t,x,1}_s)\odot \sigma_2(s,X^{t,x,2}_s)\geq 0,\quad \IP\text{-}a.s.,
\end{align}
then
\begin{align}
\label{eq:ZZineq}
Z^{t,x,1}_s \odot Z^{t,x,2}_s \geq 0,\quad 
\IP\text{-}a.s.,\quad \textrm{for any } (t,x)\in[0,T]\times\IR\textrm{ and }s\in[t,T].
\end{align}
If $g_1,\,g_2$ are strictly comonotonic and inequality \eqref{sigmaineq} holds strictly then \eqref{eq:ZZineq} is also strict.
\end{theo}
\begin{proof}
Throughout take $t\in[0,T]$, $x\in\IR$ and let $i\in\{1,2\}$. According to Theorem \ref{compilationtheorem}, for each $i\in\{1,2\}$ there exits a measurable deterministic, continuously differentiable function (in its spatial variables) $u_i:[0,T]\times \IR\to \IR$ such that $Y_s^{t,x,i}=u_i(s,X_s^{t,x,i})$  and $Z_s^{t,x,i}= (\nabla_x u_i) (s,X_s^{t,x,i})\sigma(s,X_s^{t,x,i})$ $\IP$-a.s. We have then $\IP$-a.s. that for any $s\in[t,T]$ (recall that $\sigma_i$ is a vector and $\nabla u_i$ a scalar)
\begin{align}
\nonumber
Z^{t,x,1}_s\odot Z^{t,x,2}_s &= \Big( (\nabla_x u_1)(s,X^{t,x,1}_s)\ \sigma_1(s,X^{t,x,1}_s) \Big) \odot \Big( (\nabla_x u_2) (s,X^{t,x,2}_s)\ \sigma_2(s,X^{t,x,2}_s)\Big)  \\
\label{comono-zodotz}
&= \Big(\sigma_1(s,X^{t,x,1}_s)\odot \sigma_2(s,X^{t,x,2}_s)\Big) (\nabla_x u_1)(s,X^{t,x,1}_s) (\nabla_x u)(s,X^{t,x,2}_s).
\end{align}
A standard comparison theorem for SDEs (see \cite{Protter2005}) yields that for any fixed $t$ and $T$ the mappings $x\mapsto X^{t,x,i}_T$ are increasing. This, along with the fact that $g_1$ and $g_2$ are comonotonic functions, implies that for fixed $t$ and $T$ it holds that  $x\mapsto g_1(X^{t,x,1}_T)$ and $x\mapsto g_2(X^{t,x,2}_T)$ are a.s.~comonotonic. A similar argument implies the same conclusion for the drivers $f_i$, i.e.~$x\mapsto f_1(\cdot,X^{t,x,1}_\cdot,\cdot,\cdot)$ and $x\mapsto f_2(\cdot,X^{t,x,2}_\cdot,\cdot,\cdot)$ are a.s. comonotonic.

Using the comparison theorem for quadratic BSDE (see e.g.~Theorem 2.6  in \cite{00Kob}) and the monotonicity (and comonotonicity) of $x\mapsto g_i(X^{t,x,i}_T)$ and $x\mapsto f_i(\cdot,X^{t,x,i}_\cdot,\cdot,\cdot)$ we can conclude that $x\mapsto Y^{t,x,i}$ is also a.s.~monotone. Furthermore, since $x\mapsto g_1(X^{t,x,1}_T)$, $x\mapsto g_2(X^{t,x,2}_T)$, $x\mapsto f_1(\cdot,X^{t,x,1}\cdot,\cdot,\cdot)$ and $x\mapsto f_2(\cdot,X^{t,x,2}_\cdot,\cdot,\cdot)$ are comonotonic the same comparison theorem yields that the mappings $x\mapsto Y^{t,x,1}$ and $x\mapsto Y^{t,x,2}$ are also a.s.~comonotonic. Equivalently, one can write for any $(t,x)\in[0,T]\times \IR$ that (notice that $\nabla u$ exists according to Theorem \ref{compilationtheorem})
\begin{align}
\label{aux-for-strict}
\big\langle (\nabla_x u_1)(t,x), (\nabla_x u_2) (t,x)\big\rangle \geq 0.
\end{align}
Therefore, combining \eqref{aux-for-strict} with \eqref{sigmaineq} in (\ref{comono-zodotz}) we easily obtain
\[Z^{t,x,1}_s(\omega)\odot Z^{t,x,2}_s(\omega)\geq 0, \quad \IP\text{-}a.s.\ \omega\in \Omega,\ (t,x)\in[0,T]\times\IR,\quad s\in[t,T].
\]
Under the assumption that $g_1$ and $g_2$ are strictly comonotonic it is clear that inequality \eqref{aux-for-strict} is also strict. Furthermore, if one also assumes that the inequality in \eqref{sigmaineq} holds strictly for any $(t,x)\in[0,T]\times\IR$ then \eqref{eq:ZZineq} also  holds strictly.
\end{proof}
Unfortunately it doesn't seem possible to weaken the assumptions of the previous theorem. The key factor is the representation of $Z^{t,x}$ via the function $Y^{t,x}_t=u(t,x)$ which needs to be continuously uniformly differentiable in the spatial variable and for that one needs Assumption \ref{H2} to hold.

We obtain an interesting conclusion of the previous result if we interpret the forward diffusion of the system as a backward equation. In terms of applications (as mentioned in the introduction) it is the next result that gives a condition that allows the user to conclude the positivity or negativity of the control process. 

In the next result we focus on just one FBSDE so we fix $i=1$ and we omit this index.
\begin{coro}\label{comono-theo-2}
Let the assumption of Theorem \ref{comono-theo-1} hold (fix $i=1$). Take $(t,x)\in[0,T]\times\IR$ and let $(X,Y,Z)$ be the unique solution of the FBSDE %\eqref{sde}-\eqref{bsde}
\begin{align}
\label{loc-19122008-1}
X_t&=x+\int_0^t b(s,X_s)\ud s+\int_0^t \sigma(s,X_s)\udws,\\
\label{loc-19122008-2}
Y_t &=g(X_T) +\int_t^T f(s,X_s,Y_s,Z_s)\uds-\int_t^T Z_s\udws.
\end{align}

Then, if $x\mapsto g(x)$ and $x\mapsto f(\cdot,x,,\cdot,\cdot)$ are  increasing (respectively decreasing) functions, then $Z_t \odot  \sigma(t,X_t)$  is $\IP$-a.s.~positive (respectively negative) for all $t\in[0,T]$. In particular, if the monotonicity of $g$ and $f$ (in $x$) is strict and if $\sigma$ is strictly positive then $Z$ is either strictly positive or strictly negative (according to the monotonicity of $g$ and $f$).
\end{coro}
\begin{proof}
Throughout let $x\in\IR$ and $t\in[0,T]$.
We prove the statement for the case of $g(x)$ and $f(\cdot,x,\cdot,\cdot)$ being increasing functions (in the spatial variable $x$) and we give a sketch of the proof for the decreasing case. Rewriting SDE \eqref{loc-19122008-1} as a BSDE leads to $X_t=X_T-\int_t^T b(s,X_s)\ud s-\int_t^T \sigma(s,X_s)\ud W_s$. In fact we can still rewrite the above equation in  a more familiar way, namely
\begin{align}\label{loc-19112008-3}
\tilde{Y}_t=\tilde{g}(X_T)+\int_t^T \tilde{f}(s,\tilde{Y}_s)\ud s-\int_t^T \tilde{Z}_s\ud W_s,
\end{align}
where $\tilde{Z}_s=\sigma(s,X_s)$ for $s\in[0,T]$,  $\tilde{g}(x)=x$ and $\tilde{f}(t,x,y,z)=- b(t,y)$.

At this stage we need to clarify the identification $\tilde{Z}_\cdot=\sigma(\cdot,X_\cdot)$. Let us write explicitly the dependence on the parameter $x$ of the solution $(X,\tilde{Y},\tilde{Z})$ of the FBSDE (\ref{loc-19122008-1}), (\ref{loc-19112008-3}), i.e. we write $(X,\tilde{Y} ,\tilde{Z})$ to denote $(X,\tilde{Y},\tilde{Z})$.
Note that the solution of the BSDE (\ref{loc-19112008-3}) is the solution of SDE (\ref{loc-19122008-1}) which is a Markov process. We can then write $\tilde{Y}_\cdot= X_\cdot=\tilde{u}(\cdot,X_\cdot)$ where $\tilde{u}$ is the identity function (infinitely differentiable). Under Assumption \ref{H1} both $X$ an $\tilde{Y}$ are differentiable as a functions of $x$ (see Theorem \ref{compilationtheorem}), we have then  $\tilde{Z}_\cdot=(\nabla_x \tilde{u}) (\cdot,X_\cdot)\sigma(\cdot,X_\cdot)$. And since $\tilde{u}$ is the identity function with derivative being the constant function $1$, it follows immediately that $\tilde{Z}_\cdot=\sigma(\cdot,X_\cdot)$.

Our aim is to use the previous theorem to imply this result. So we only have to check that its assumptions are verified. Comparing the terminal conditions of (\ref{loc-19122008-2}) and (\ref{loc-19112008-3}), i.e. comparing $x\mapsto g(X^x_T)$ with $x\mapsto \tilde{g}(X^x_T)= X^x_T$ it is clear that both functions are almost surely increasing. Further, the driver function $\tilde{f}$ of BSDE \eqref{loc-19112008-3} is given by $\tilde{f}(t,x,y,z)=\tilde{f}(t,y)=-b(t,y)$ which is independent of $x$. Clearly $x\mapsto \tilde{f}(\cdot,x,\cdot,\cdot)$ and $x\mapsto f(\cdot,x,\cdot,\cdot)$ are comonotonic.

Theorem \ref{comono-theo-1} applies and we conclude immediately that
\begin{align}
\label{Zodotsigma}
Z_t\odot \tilde{Z}_t=Z_t\odot \sigma(t,X_t)\geq 0,~\IP\text{-}a.s.\quad t\in[0,T].
\end{align}

For the other case, when $g$ is a decreasing function, the approach is very similar. We rewrite the SDE (\ref{loc-19122008-1}) in the following way, 
\[
-X_t=-X_T+\int_t^T b(s,X_s)\ud s-\int_t^T\big[ -\sigma(s,X_s)\big]\ud W_t,\quad t\in[0,T].
\]
The terminal condition of the above BSDE is given by $x\mapsto \tilde{g}(x)=-x$ evaluated at $x=X_T$ and the driver $\tilde{f}(t,x,y,z)=\tilde{f}(t,y)=b(t,-y)$ which is independent of $x$. Since $\tilde{g}$ is a decreasing function, we obtain our result by comparing the above BSDE with (\ref{loc-19122008-2}) and applying the previous theorem.
\end{proof}
The above corollary allows one to conclude in particular the strict positivity of the control process. If one is only interested in establishing positivity (ignoring strictness), then one can indeed lower the strength of the assumptions.
\begin{lemma}
Let Assumption \ref{H1} hold and $m=1$. Assume that $x\mapsto g(x)$ and $x\mapsto f(\cdot,x,0,0)$ are both monotone increasing then $ Z_t\odot \sigma(t,X_t)\geq 0,$ $\IP$-a.s. for all $t\in[0,T]$. If $x\mapsto g(x)$ and $x\mapsto f(\cdot,x,\cdot,\cdot)$ are both monotone decreasing then $Z_t\odot \sigma(t,X_t)\leq 0$, $\IP$-a.s. for all $t\in[0,T]$.
\end{lemma}
\begin{remark}
Again, as in Remark \ref{caseofmdiff1}, it is possible to state and prove the same result for $m\geq 1$. One needs to impose that $\sigma$ is a square diagonal matrix and the SDE to be a decoupled system.
\end{remark}
\begin{remark}
It is possible to weaken the assumptions of this lemma as was done for Theorem 4.3.6 in \cite{reis2011} or Corollary 2 in \cite{IdRZ2010}. Namely, the conditions are weakened to Lipschitz type conditions with the appropriate Lipschitz ``constant'', then one argues similarly but combining with a regularization argument.
\end{remark}
\begin{proof}
Throughout let $t\in[0,T]$ and $x\in\IR$. Then due to the representation formulas in \eqref{representation} we have $\IP$-a.s. that
\begin{align}
\label{trick}
Z_t\odot\sigma(t,X_t)
= D_t Y_t\odot \sigma(t,X_t)  
= \nabla_x Y_t (\nabla_x X_t)^{-1} \sigma(t,X_t) \odot\sigma(t,X_t).
\end{align}
It is trivial to verify that $\sigma(t,X_t) \odot\sigma(t,X_t)\geq 0$. It remains to establish a result concerning the sign of $\nabla_x Y$ and $(\nabla_x X)^{-1}$.

Under the assumptions it is easy to verify that the solution of \eqref{nablasde} is positive. The solution of $\nabla_x X$ is essentially a positive geometric Brownian motion with a nonlinear drift and volatility which in turn implies that $(\nabla_x X)^{-1}$ is also positive. If we manage to deduce a result concerning the sign of $\nabla_x Y$ we are then able to obtain a weaker version of Corollary \ref{comono-theo-2}. 

The methodology developed to deduce moment and a priori estimates for quadratic BSDE and illustrated in Lemma 3.1 and 3.2 of \cite{IdR2010} (or Chapter 2 in \cite{reis2011}) allow the following equality
\begin{align}
\label{simplifiedeq}
\nabla_x Y_t = \IE^{\widehat{\IP}}\big[
e_T (e_t)^{-1} \nabla_x g(X_t) \nabla X_T
+\int_t^T [e_r e_t^{-1} (\nabla_x f)(r,X_r,0,0)\nabla_x X_r] \ud r
\big|\cF_t\big],
\end{align}
where the process $e$ and the measure $\widehat{\IP}$ are defined as
\[
e_t=\exp\Big\{ \int_0^t \frac{f(r,X_r,Y_r,Z_r)-f(r,X_r,0,Z_r)}{Y_r}\1_{Y_r\neq 0} \ud r \Big\},
\]
and $\widehat{\IP}$ is a probability measure with Radon-Nikodym density given by
\[
\frac{\ud \widehat{\IP}}{\ud \IP}=M_T=\cE\Big( \int_0^T \frac{f(r,X_r,0,Z_r)-f(r,X_r,0,0)}{|Z_r|^2}Z_r\1_{|Z_r|\neq 0} \ud W_r \Big).
\]
Both $(e_t)_{t\in[0,T]}$ and $M$ are well defined. The first because $y\mapsto f(\cdot,\cdot,y,\cdot)$ is assumed to be uniformly Lipschitz and hence $e$ is bounded from above and below and away from zero. The second follows from a combination of the growth assumptions on $\nabla_z f$ and the fact that $\int Z \ud W$ is a bounded mean oscillation martingale\footnote{This observation is key in many results for quadratic BSDE. The stochastic exponential of a BMO martingale is uniformly integrable and defines a proper density. This type of reasoning can be found ubiquitously in \cite{AIdR07} or \cite{IdR2010} for example.} (BMO).

We have already seen that $\nabla X$ is positive and it also trivial to conclude that the process $e$ also is. Given that $g$ and $f$ are differentiable, then saying that these functions are monotonic (in x) boils down to making a statement on the sign of $(\nabla_x g)(x)$ and $(\nabla_x f)(\cdot,x,0,0)$. If one assumes that $g$ and $f(\cdot,x,0,0)$ are monotone increasing in $x$ then $(\nabla g)(x)\geq 0$ and $(\nabla_x f)(\cdot,x,0,0)\geq 0 $ for all $x$. Hence from \eqref{simplifiedeq} (and the remarks above) we conclude that $\nabla_x Y$ is also positive. Returning to \eqref{trick} we have then that $Z_t \odot \sigma(t,X_t)\geq 0$ which proves our result.

The arguments are similar for the case when $g(x)$ and $f(\cdot,x,0,0)$ are
decreasing functions. 
\end{proof}

{\bf Acknowledgments:} The first author would like to thank Peter Imkeller and Ulrich Horst for their comments. The first author gratefully acknowledges the partial support from the CMA/FCT/UNL through project PEst-OE/MAT/UI0297/2011. 

This work was partially supported by the project SANAF  UTA\_{}CMU/MAT/0006/2009.

\bibliographystyle{abbrv}
%\bibliography{refcomono}        % references.bib is the name of the database

\begin{thebibliography}{10}

\bibitem{AIdR07}
S.~Ankirchner, P.~Imkeller, and G.~dos Reis.
\newblock Classical and variational differentiability of {BSDE}s with quadratic
  growth.
\newblock {\em Electron. J. Probab.}, 12(53):1418--1453, 2007.

\bibitem{10AIdR}
S.~Ankirchner, P.~Imkeller, and G.~dos Reis.
\newblock Pricing and hedging of derivatives based on non-tradable underlyings.
\newblock {\em Math. Finance}, 20(2):289 -- 312, 2010.

\bibitem{MR2134722}
F.~Antonelli and A.~Kohatsu-Higa.
\newblock Densities of one-dimensional backward {SDE}s.
\newblock {\em Potential Anal.}, 22(3):263--287, 2005.

\bibitem{05CKW}
Z.~Chen, R.~Kulperger, and G.~Wei.
\newblock A comonotonic theorem for {BSDE}s.
\newblock {\em Stochastic Process. Appl.}, 115(1):41--54, 2005.

\bibitem{freidosreis2011}
C.~Frei and G.~dos Reis.
\newblock Quadratic FBSDE with generalized Burgers' type nonlinearities, perturbations and large deviations.
\newblock {\em Preprint}, 2011.

\bibitem{HPdR10}
U.~Horst, T.~Pirvu, and G.~dos Reis.
\newblock On securitization, market completion and equilibrium risk transfer.
\newblock {\em Math. Financ. Econ.}, 2(4):211--252, 2010.

\bibitem{HIM2005}
Y.~Hu, P.~Imkeller, and M.~M{\"u}ller.
\newblock Utility maximization in incomplete markets.
\newblock {\em Ann. Appl. Probab.}, 15(3):1691--1712, 2005.

\bibitem{IdR2010}
P.~Imkeller and G.~dos Reis.
\newblock Path regularity and explicit convergence rate for {BSDE} with
  truncated quadratic growth.
\newblock {\em Stochastic Process. Appl.}, 120:348--379, 2010.

\bibitem{IdRZ2010}
P.~Imkeller, G.~dos Reis and J. Zhang.
\newblock {R}esults on numerics for {FBSDE} with drivers of quadratic growth
\newblock {\em in: Contemporary Quantitative Finance}, Springer, 159-182, 2010

\bibitem{EPQ}
N.~E. Karoui, S.~Peng, and M.~Quenez.
\newblock Backward stochastic differential equations in finance.
\newblock {\em Math. Finance}, 7(1):1--71, 1997.

\bibitem{kazamaki}
N.~Kazamaki.
\newblock {\em Continuous Exponential martingales and {BMO}}, volume 1579 of
  {\em Lecture Notes in Mathematics}.
\newblock Springer-Verlag, 1994.

\bibitem{00Kob}
M.~Kobylanski.
\newblock Backward stochastic differential equations and partial differential
  equations with quadratic growth.
\newblock {\em Ann. Probab.}, 28(2):558--602, 2000.

\bibitem{Protter2005}
P.~E. Protter.
\newblock {\em Stochastic integration and differential equations}.
\newblock Applications of Mathematics (New York). Springer-Verlag, 2nd edition,
  2005.
\newblock Version 2.1.

\bibitem{reis2011}
G.~Dos~Reis.
\newblock {\em Some advances on quadratic {BSDE}: {T}heory - {N}umerics -
  {A}pplications}.
\newblock LAP LAMBERT Academic Publishing, May 2011.
\newblock ISBN: 978-3-844333077.
\end{thebibliography}

\end{document}